\documentclass[10pt,reqno]{amsart}
\usepackage{amsmath}
\usepackage{amssymb}
\usepackage{amsthm}
\usepackage{eepic,epic}
\usepackage{epsfig}
\usepackage{graphicx}

\newlength{\defbaselineskip}
\newcommand{\setlinespacing}[1]%
           {\setlength{\baselineskip}{#1 \defbaselineskip}}

\numberwithin{equation}{section}

\newtheorem{thm}{Theorem}[section]
\newtheorem{cor}[thm]{Corollary}
\newtheorem{lem}[thm]{Lemma}
\newtheorem{prop}[thm]{Proposition}

\theoremstyle{definition}

\theoremstyle{remark}
\newtheorem{rem}[thm]{Remark}
\numberwithin{equation}{section}

\begin{document}

\title[Well-posedness for the INLS]
{On well-posedness for the inhomogeneous nonlinear Schr\"{o}dinger equation in the critical case}

\author{Jungkwon Kim, Yoonjung Lee and Ihyeok Seo}

\thanks{This research was supported by NRF-2019R1F1A1061316.}


\address{Department of Mathematics, Sungkyunkwan University, Suwon 16419, Republic of Korea}
\email{kimjk809@skku.edu}

\email{yjglee@skku.edu}

\email{ihseo@skku.edu}

\begin{abstract}
In this paper we study the well-posedness for the inhomogeneous nonlinear Schr\"odinger equation
$i\partial_{t}u+\Delta u=\lambda|x|^{-\alpha}|u|^{\beta}u$ in Sobolev spaces $H^s$, $s\geq0$.
The well-posedness theory for this model has been intensively studied in recent years,
but much less is understood compared to the classical NLS model where $\alpha=0$.
The conventional approach does not work particularly for the critical cases $\beta=\frac{4-2\alpha}{d-2s}$. It is still an open problem.
The main contribution of this paper is to develop the well-posedness theory in this critical case (as well as non-critical cases).
To this end, we approach to the matter in a new way based on a weighted $L^p$ setting
which seems to be more suitable to perform a finer analysis for this model.
This is because it makes it possible to handle
the singularity $|x|^{-\alpha}$ in the nonlinearity more effectively.
This observation is a core of our approach that covers
the critical case successfully.
\end{abstract}

\maketitle

\section{Introduction}
In this paper we are concerned with the Cauchy problem for the inhomogeneous nonlinear Schr\"odinger equation (INLS)
\begin{equation}\label{INLS}
\left\{
\begin{aligned}
&i \partial_{t} u + \Delta u =\lambda  |x|^{-\alpha} |u|^{\beta} u,\quad (x,t) \in \mathbb{R}^d \times \mathbb{R}, \\
&u(x, 0)=u_0(x),
\end{aligned}
\right.
\end{equation}
where $\alpha,\beta>0$ and $ \lambda = \pm 1$.
The equation is called \textit{focusing} INLS when $\lambda=-1$ and \textit{defocusing} INLS when $\lambda=1$.
This type of equation arises in nonlinear optics and plasma physics
for the propagation of laser beams in an inhomogeneous medium (\cite{B,TM}).
The case $\alpha = 0$ is the classical nonlinear Schr\"{o}dinger equation (NLS)
whose well-posedness theory in Sobolev spaces $H^s$ has been extensively studied over the past several decades
and is well understood (\cite{CW, CW2,GV3,GV4,K,Ts}).
However, much less is known about the INLS equation which has drawn attention in recent years.

Before reviewing known results for the Cauchy problem \eqref{INLS},
we recall the critical Sobolev index from which one can divide the matter into three cases.
Note first that if $u(x,t)$ is a solution of \eqref{INLS}
so is $u_{\lambda}(x,t)= \lambda^{\frac{2-\alpha}{\beta}} u(\lambda x, \lambda^2 t)$, with the initial data $u_{\lambda,0}(x)=u_\lambda(x,0)$
for all $\lambda>0$.
We then easily see
$$\|u_{\lambda,0}\|_{\dot{H}^s} = \lambda^{s+\frac{2-\alpha}{\beta}-\frac{d}{2}}\|u_0 \|_{\dot{H}^s}$$
from which the critical Sobolev index is given by $s_{c}= d/2-(2-\alpha)/\beta$
which determines the scale-invariant Sobolev space $\dot{H}^{s_c}$.
In this regard, the case $s_{c}=0$ (alternatively $\beta =\frac{4-2\alpha}{d}$) is referred to as the mass-critical (or $L^2$-critical).
If $s_{c}=1$ (alternatively $\beta =\frac{4-2\alpha}{d-2}$) the problem is called the energy-critical (or $H^1$-critical),
and finally it is known as the mass-supercritical and energy-subcritical if $0 < s_{c} < 1$.
Therefore, we restrict our attention to the cases where $0\leq s_c\leq1$.
We also assume $d\geq3$ to make the review shorter.

The well-posedness for the INLS equation \eqref{INLS} was first studied by Genoud-Stuart \cite{GS} in the sense of distribution.
Using the abstract theory developed by Cazenave \cite{C},
they showed that the focusing INLS equation with $0<\alpha<2$ is locally well-posed in $H^1(\mathbb{R}^d)$
if $0<\beta<\frac{4-2\alpha}{d-2}$.
In this case, Genoud \cite{Ge} and Farah \cite{Fa} also showed how small should be the initial data to have global well-posedness,
respectively, in the spirit of Weinstein \cite{We} and Holmer-Roudenko \cite{HR} for the classical case $\alpha=0$.

In all of the previous results mentioned above, the solution has been constructed in the energy space $H^1$ smaller than $L^2$.
By comparison, Guzm\'{a}n \cite{GU} used the standard contraction mapping argument based on the known classical Strichartz estimates
to establish the well-posedness for the INLS equation \eqref{INLS} in various Sobolev spaces $H^s$, $s\leq1$.
His result in $H^1$ was also improved by Dinh \cite{D} extending the validity of $\alpha$ in low dimensions.
More importantly, Guzm\'{a}n showed that \eqref{INLS} with $0<\alpha<2$
is locally well-posed in $L^2$ larger than the energy space $H^1$ if $0<\beta<\frac{4-2\alpha}{d}$.
He also treated the local well-posedness in $H^s$ for $\max\{0,s_c\}<s\leq1$
if $0<\beta<\frac{4-2\alpha}{d-2s}$.
But here, the critical case $\beta=\frac{4-2\alpha}{d-2s}$ with $0\leq s\leq 1$ and $d\geq3$ is left unsolved.
Unlike the NLS equation where $\alpha=0$ (\cite{CW, CW2}), the conventional approach does not work for this case.
It is still an open problem.

The main contribution of this paper is to develop the well-posedness theory in this critical case
as well as the non-critical case.
To this end, we approach to the matter in a new way based on a weighted $L^p$ setting
which does seem to be more suitable to perform a finer analysis for the INLS model.
This is because it makes it possible to deal with
the singularity $|x|^{-\alpha}$ in the nonlinearity more effectively.
This observation is a core of our approach that covers
the critical case successfully.

Before stating our results, we introduce the following weighted space-time norms
\begin{equation}
\|f\|_{L_{t}^q L_{x}^r(|x|^{-r\gamma})}=  \left( \int_{\mathbb{R}} \left(\int_{\mathbb{R}^d}  |x|^{-r\gamma} |f(x)|^r  dx \right)^{\frac{q}{r}} dt\right)^{\frac{1}{q}} \nonumber
\end{equation}
where $1\leq r< \infty$ and $\gamma\geq0$.
Our first result is then the following local well-posedness for the INLS equation \eqref{INLS}
in $L^2$ up to the $L^2$-critical case, $\beta=(4-2\alpha)/d$.

\begin{thm}\label{Thm1}
Let $d \geq 3$, $0 < \alpha < 2 $ and $0 < \beta \leq (4-2\alpha)/d$.
Then for $u_0\in L^2(\mathbb{R}^d)$ there exist\, $T>0$ and a unique solution
$u \in C([0,T] ; L^{2}) \cap L^{q}([0,T] ; L^{r}(|x|^{-r \gamma}))$
of the problem \eqref{INLS} for
\begin{equation*}\label{condition23}
\max \lbrace 0,\, \frac{\alpha-1}{\beta+1} \rbrace <\gamma < \min \{ 1,\, \frac{\alpha}{\beta+1} \}
\end{equation*}
and $(q, r)$ satisfying
\begin{equation*}\label{condition21}
\frac{2}{q}=d( \frac{1}{2}-\frac{1}{r}) +\gamma, \quad \frac{\gamma}{2} < \frac{1}{q} \leq \frac{1}{2},
\end{equation*}	
\begin{equation*}\label{condition22}
\frac{1}{2(\beta+1)}<\frac{1}{r}
<\frac{d-2(\alpha-1)}{2d(\beta+1)}+\frac{\gamma}{d}.
\end{equation*}
Furthermore, the continuous dependence on initial data holds.
\end{thm}

As we shall see in the proof of the theorem, we can give a precise estimate for the life span of the solution
according to the size of the initial data, $T\sim\|u_0\|_{L^2}^{-4\beta/(4-2\alpha-d\beta)}$, in the subcritical case.
Thanks to the mass conservation for the INLS equation,
\begin{equation}\label{mass}
	\textit{Mass} \equiv M[u(t)] = \int_{\mathbb{R}^d} |u(x, t)|^2 dx = M[ u_0 ],
\end{equation}
we can then apply the local result repeatedly, preserving the length of the time interval to get a global solution.
However, the situation for the critical case is quite different.
In this case, the local solution exists in a time interval depending on the data $u_0$ itself and not on its norm.
Thus, the conservation \eqref{mass} does not guarantee the existence of a global solution any more.
For this reason, $\|u_0\|_{L^2}$ is assumed to be small for the critical case in the following global result.

\begin{cor}\label{cor}
Under the same conditions as in Theorem \ref{Thm1},
the local solution extends globally in time with
$u \in C([0,\infty) ; L^2) \cap L^{q}([0,\infty) ; L^{r}(|x|^{-r \gamma}))$
for $u_0\in L^2(\mathbb{R}^d)$.
Particularly in the critical case $\beta=(4-2\alpha)/d$, $\|u_0\|_{L^2}$ is assumed to be small and
the solution scatters in $L^2$, i.e., there exists $\varphi\in L^2$ such that
$$\lim_{t\rightarrow\infty}\|u(t)-e^{it\Delta}\varphi\|_{L_x^2}=0.$$
\end{cor}

When $\alpha = 0$, the INLS model becomes the classical NLS equation which has been the subject of intensive work for a long time.
In this case, the $L^2$ theory was obtained by Tsutsumi \cite{Ts} when $0<\beta<4/d$.
The $L^2$-critical case $\beta=4/d$ was later treated by Cazenave-Weissler \cite{CW}.
Particularly when $\alpha=0$ in our approach, it is deduced that $\gamma=0$ (see Remark \ref{opl}),
and in this case resulting results cover this classical results for the NLS equation.

We also obtain the following local well-posedness in $H^s$ for $s>0$ up to the critical case $\beta=(4-2\alpha)/(d-2s)$
which is equivalent to $s_c=s$.

\begin{thm}\label{Thm2}
Let $d \geq 3$ and $0<s<1/3$. Assume that
\begin{equation*}
\max\{\frac{26-3d}{12},\, \frac{12s+4ds-8s^2}{d+4s}\}<\alpha<2
\end{equation*}
and
\begin{equation*}
\max\{0,\, \frac{10s-2\alpha}{d-6s}\}< \beta \leq \frac{4-2\alpha}{d-2s}.
\end{equation*}
Then for $u_0\in H^s(\mathbb{R}^d)$ there exist\, $T>0$ and a unique solution
$u \in C([0,T] ; H^{s}) \cap L^{q}([0,T] ; L^{r}(|x|^{-r \gamma}))$
of the problem \eqref{INLS} if
\begin{equation*}\label{condition33}
\max \{3s,\, \frac{\alpha+s-1}{\beta+1}\} < \gamma < \min \{1+s,\, \frac{\alpha-s}{\beta+1},\ \frac{d\beta+2\alpha-4s}{2(\beta+1)} \}
\end{equation*}
and $(q, r)$ satisfies
\begin{equation*}\label{condition31}
\frac{2}{q}=d( \frac{1}{2}-\frac{1}{r}) +\gamma-s, \quad \frac{\gamma-s}{2}< \frac{1}{q} \leq \frac{1}{2},
\end{equation*}	
\begin{equation*}\label{condition32}
\max \lbrace \frac{1}{2(\beta+1)},\,\frac{1}{2(\beta+1)} + \frac{2s-\alpha}{d(\beta+1)} + \frac{\gamma}{d}  \rbrace <  \frac{1}{r} < 
 \frac{d-2s-2(\alpha-1)}{2d(\beta+1)} + \frac{\gamma}{d}.
\end{equation*}
Furthermore, the continuous dependence on initial data holds.
\end{thm}

\begin{rem}
When $\beta\leq(4-2\alpha)/d$, we already have the well-posedness result in $L^2$ which implies automatically
a result in $H^s$. Hence an essential part in the latter case is when $\beta>(4-2\alpha)/d$.
The range of $\beta$ in the theorem contains this range when $\alpha\geq5s$,
and when $\alpha<5s$ if we further assume $\alpha\geq\frac{5ds-2d+12s}{6s}$.
\end{rem}

One of the most basic tools for the well-posedness of nonlinear dispersive equations is
the contraction mapping principle.
The key ingredient in this argument is the availability of Strichartz estimates.
Indeed, Guzm\'{a}n \cite{GU} makes use of the fractional product and chain rules to derive a contraction from the nonlinearity in \eqref{INLS}
based on the classical Strichartz estimates.
Unfortunately, this approach does not work for the critical case.
In our approach we make use of the weighted Strichartz estimates (see \eqref{T} and \eqref{TT*})
which deal with the singularity $|x|^{-\alpha}$ in the nonlinearity more effectively.
More interestingly, we take advantage of the smoothing estimates (see \eqref{T*} or \eqref{smoothing})
which makes it possible to apply more directly the contraction mapping principle without using the fractional product and chain rules.
As consequences, we can handle the $H^s$-critical cases, and even $u \in L_t^{q}L_x^{r}(|x|^{-r \gamma}))$ in which $L_x^r$ is less restrictive than the usual $H_x^{s,r}$.
Indeed, there are no less restrictive alternatives than $L_x^r$ within the framework of Strichartz estimates.
These aspects are a core of our approach which establishes the well-posedness theory up to the critical case successfully.

\begin{prop}\label{Prop1}
Let $d \geq 3 $ and $0\leq s<1/2$.
Assume that $(q,r)$ and $(\tilde{q}, \tilde{r})$ satisfy
\begin{equation}\label{adm}
\frac{2}{q}=d\left(\frac{1}{2}-\frac{1}{r}\right)+\gamma -s,\quad \frac{\gamma-s}{2}< \frac{1}{q} \leq \frac{1}{2},
\quad \frac{\gamma-s}{2} \leq \frac{1}{r} <\frac{1}{2},
\end{equation}
\begin{equation}\label{adm2}
\frac{2}{\tilde{q}}=d\left(\frac{1}{2}-\frac{1}{\tilde{r}}\right)+\tilde{\gamma}+s,
\quad \frac{\tilde{\gamma}+s}{2}< \frac{1}{\tilde{q}} \leq \frac{1}{2},
\quad \frac{\tilde{\gamma}+s}{2}\leq  \frac{1}{\tilde{r}} < \frac{1}{2},
\end{equation}
for $3s < \gamma < 1+s$ and $s < \tilde{\gamma}<1-s$.
Then we have
\begin{equation}\label{T}	
\big\| e^{it\Delta} f \big\|_{L_t^{q} L_{x}^{r} (|x|^{-r\gamma})} \lesssim \| f\|_{\dot{H}^s},
\end{equation}
\begin{equation}\label{T*}
\bigg\| \,|\nabla|^s\int_{-\infty}^\infty  e^{-i\tau\Delta} F(\cdot, \tau)d\tau \bigg\|_{L^2 } \lesssim \| F \|_{L_t^{\tilde{q}'}L_x^{\tilde{r}'} (|x|^{\tilde{r}'\tilde{\gamma}})},
\end{equation}
and if we additionally assume $q>\tilde{q}'$
\begin{equation}\label{TT*}
\bigg\| \int_{0} ^{t} e^{i(t-\tau)\Delta} F(\cdot,\tau)d\tau \bigg\|_{L_t^{q}L_x^{r} (|x|^{-r\gamma})} \lesssim
\| F \|_{L_t^{\tilde{q}'}L_x^{\tilde{r}'} (|x|^{\tilde{r}'\tilde{\gamma}})}.
\end{equation}
\end{prop}

\begin{rem}\label{opl}
Particularly when $s=0$, the above three estimates also hold for the case where $\gamma=0$ and $\tilde{\gamma}=0$,
including $(q,r)=(\infty,2)$ and $(\tilde{q},\tilde{r})=(\infty,2)$.
The first two estimates \eqref{T} and \eqref{T*} also hold for $(q,r)=(2,2)$ with $\gamma=1+s$ and $(\tilde{q},\tilde{r})=(2,2)$ with $\tilde{\gamma}=1-s$, respectively.
Using these cases in our approach, one is led to have the theorems up to the boundary points in some ranges of $q,r,\gamma$.
But this is an insignificant part, and hence we omit the details.
\end{rem}

\begin{rem}
The first conditions in \eqref{adm} and \eqref{adm2} are just the scaling conditions for which the estimates hold.
\end{rem}

It should be also noted that \eqref{T*} is equivalent to
\begin{equation}\label{smoothing}
\big\|\,|\nabla|^s e^{it\Delta} f \big\|_{L_t^{\tilde{q}} L_{x}^{\tilde{r}} (|x|^{-\tilde{r}\tilde{\gamma}})} \lesssim \| f \|_{L^2}
\end{equation}
by duality. This shows the smoothing effect of the Schr\"{o}dinger propagator.
To obtain the well-posedness in $H^s$, $s>0$, we make use of this type of smoothing estimates
because it allows us to have the inhomogeneous estimate \eqref{TT*} without derivative
from which it is possible to apply directly the contraction mapping principle to the nonlinearity in \eqref{INLS}
without using the fractional product and chain rules.
Indeed, by the standard $TT^*$ argument, \eqref{T} and \eqref{T*} with the same $s$ implies \eqref{TT*} without derivative.
In this regard, we need to have the common range of $s$, $0\leq s<1/2$, for which both \eqref{T} and \eqref{T*} hold,
although \eqref{T} holds more widely for $0\leq s<(d-2)/2$ (see Section \ref{sec2} for details).

\

\noindent\textit{Outline of paper.}
In Section \ref{sec2}, we prove Proposition \ref{Prop1} by making
use of the complex interpolation between the classical Strichartz estimates and
the Kato-Yajima smoothing estimates
appealing to the complex interpolation space identities.
In Section \ref{sec3}, we obtain some weighted estimates for the nonlinear term $|x|^{-\alpha}|u|^{\beta}u$ of the INLS equation.
These nonlinear estimates will play a crucial role in later sections \ref{sec4} and \ref{sec5} to obtain the well-posedness results
by applying the contraction mapping argument to the nonlinearity
along with the estimates in Proposition \ref{Prop1}.

Throughout this paper, the letter $C$ stands for a positive constant which may be different
at each occurrence. We also denote $A\lesssim B$ to mean $A\leq CB$
with unspecified constants $C>0$.

\section{Proof of Proposition \ref{Prop1}}\label{sec2}
In this section we prove the estimates in Proposition \ref{Prop1}.
Let $d\geq3$.
We then first recall the classical Strichartz estimate
\begin{equation}\label{clastr}
	\| e^{it\Delta} f \|_{L_t^{a} L_x^{b}} \lesssim \| f \|_{L^{2}}
\end{equation}
which holds if and only if $\frac{2}{a}=d \left( \frac{1}{2}-\frac{1}{b} \right)$ and $2\leq a, b\leq \infty$.
This estimate was first established by Strichartz \cite{St} in the diagonal case $a=b$
and then extended to mixed space-time norms as above (see \cite{GV4, KT}).
It is also necessary for us to make use of the Kato-Yajima smoothing estimate
\begin{equation}\label{weism}
	\| \,|\nabla|^\rho e^{it\Delta} f \|_{L_{t,x}^{2}(|x|^{-2(1-\rho)})} \lesssim \| f \|_{L^{2}}
\end{equation}
which holds if and only if $-\frac{d-2}{2} < \rho < \frac{1}{2}$.
This estimate was first discovered by Kato and Yajima \cite{KY} for $0\leq \rho<\frac12$
and Ben-Artzi and Klainerman \cite{BK} gave an alternate proof of this result.
Since then, the full range was obtained by Sugimoto \cite{Su},
although it was later shown by Vilela \cite{V} that the range is indeed optimal (see also \cite{W}).

Now we make use of the complex interpolation between these two estimates \eqref{clastr} and \eqref{weism}
by appealing to the following complex interpolation space identities (see \cite{BL}).

\begin{lem}\label{idd}
Let $0<\theta<1$, $1\leq p_0,p_1<\infty$ and $s_0,s_1\in\mathbb{R}$.
Then the following identities hold:
\begin{itemize}
\item Given two complex Banach spaces $A_0,A_1$,
$$(L^{p_0}(A_0),L^{p_1}(A_1))_{[\theta]}=L^p((A_0,A_1)_{[\theta]}),$$
and
$$(L^{p_0}(w_0),L^{p_1}(w_1))_{[\theta]}=L^p(w),$$
with $1/p=(1-\theta)/p_0+\theta/p_1$ and $w=w_0^{p(1-\theta)/p_0}w_1^{p\theta/p_1}$.
\item
$$(\dot{H}^{s_0},\dot{H}^{s_1})_{[\theta]}=\dot{H}^s$$
with $s=(1-\theta)s_0+\theta s_1$, $s_0\neq s_1$.
\end{itemize}
Here, $(\cdot\,,\cdot)_{[\theta]}$ denotes the complex interpolation functor.
\end{lem}

Using the complex interpolation between \eqref{clastr} and \eqref{weism}, we first see
\begin{equation*}
\| e^{it\Delta} f \|_{(L_t^{a}L_x^{b},L_t^{2}L_x^{2}(|x|^{-2(1-\rho)}))_{[\theta]}}
\lesssim \| f \|_{(\dot{H}^0,\dot{H}^{-\rho})_{[\theta]}},
\end{equation*}
and then we make use of Lemma \ref{idd} to get
\begin{equation}\label{inter}
\| e^{it\Delta} f \|_{L_t^{q} L_x^{r}(|x|^{-r\gamma})} \lesssim \left\| f \right\|_{\dot{H}^{-\sigma}}
\end{equation}
where
\begin{equation}\label{cdo}
\frac1q=\frac{1-\theta}{a}+\frac{\theta}{2},\quad \frac1r=\frac{1-\theta}{b}+\frac{\theta}{2},
\quad\gamma=(1-\rho)\theta,\quad \sigma=\rho\theta
\end{equation}
under the conditions
\begin{equation}\label{under}
 2\leq a,b\leq \infty,\ (a,b)\neq(\infty,2),\ \frac{2}{a}=d \left( \frac{1}{2}-\frac{1}{b} \right),\ 0<\theta<1,\ -\frac{d-2}{2} < \rho < \frac{1}{2}.
\end{equation}
By eliminating the redundant exponents $a,b,\rho,\theta$ here,
all the conditions on $q,r,\gamma,\sigma$ for which \eqref{inter} holds are summarized as follows:
\begin{equation}\label{inttt0}
0< \gamma+\sigma < 1,\quad -\frac{\gamma(d-2)}{d} < \sigma < \gamma,
\end{equation}
\begin{equation}\label{intecon0}
\frac{2}{q}=d \left( \frac{1}{2}-\frac{1}{r} \right)+\gamma+\sigma,
\quad \frac{\gamma+\sigma}{2} \leq \frac{1}{q}, \frac{1}{r} \leq \frac{1}{2},\quad(\frac1q,\frac1r)\neq(\frac{\gamma+\sigma}{2},\frac12).
\end{equation}
Indeed, from the last two conditions in \eqref{cdo} it follows that $\theta=\gamma+\sigma$.
The third condition in \eqref{under} implies $2/q=d(1/2-1/r)+\theta$.
Hence the first conditions in \eqref{inttt0} and \eqref{intecon0} are concluded.
Combining the last conditions in \eqref{cdo} and \eqref{under}, we see $-(d-2)\theta/2<\sigma<\theta/2$,
and then substituting  $\theta=\gamma+\sigma$ implies the second condition in \eqref{inttt0}.
Finally, applying the first two conditions in \eqref{under} with $\theta=\gamma+\sigma$ to the first two ones in \eqref{cdo},
the last two conditions in \eqref{intecon0} are concluded.

Now let $s\geq0$.
By substituting $-s$ for $\sigma$ in \eqref{inter}, we obtain the first estimate \eqref{T} in the proposition
because \eqref{intecon0} implies directly the condition \eqref{adm} and
\eqref{inttt0} implies $\frac{d}{d-2}s<\gamma<1+s$ which is wider than $3s<\gamma<1+s$.
On the other hand, we substitute $s$, $\tilde{q}$, $\tilde{r}$, $\tilde{\gamma}$ for $\sigma$, $q$, $r$, $\gamma$,
respectively in \eqref{inter} to obtain \eqref{smoothing} which is equivalent to the second estimate \eqref{T*} in the proposition.
In this case, \eqref{intecon0} implies directly the condition \eqref{adm2}
and it is not difficult to see that \eqref{inttt0} implies the condition $s<\tilde{\gamma}<1-s$.
Hence the common range of $s$ for which both \eqref{T} and \eqref{T*} hold is given by $0\leq s<1/2$.
This is the reason why we choose $3s<\gamma<1+s$ instead of $\frac{d}{d-2}s<\gamma<1+s$ as the range of $\gamma$ for which \eqref{T} holds.

It remains to prove \eqref{TT*}.
By the $TT^*$ argument, \eqref{T} and \eqref{T*} with the same value of $s$ imply
$$\bigg\| \int_{-\infty}^\infty e^{i(t-\tau)\Delta} F(\cdot,\tau)d\tau \bigg\|_{L_t^{q}L_x^{r} (|x|^{-r\gamma})} \lesssim
\| F \|_{L_t^{\tilde{q}'}L_x^{\tilde{r}'} (|x|^{\tilde{r}'\tilde{\gamma}})}. $$
By applying the following Christ-Kiselev lemma \cite{CK}, we now get \eqref{TT*} for $q>\tilde{q}'$ as desired.

\begin{lem}
Let $X$ and $Y$ be two Banach spaces and let $T$ be a bounded linear operator from $L^\alpha(\mathbb{R};X)$ to $L^\beta(\mathbb{R};Y)$
such that
$$Tf(t)=\int_{\mathbb{R}} K(t,s)f(s)ds.$$
Then the operator
$$\widetilde{T}f(t)=\int_{-\infty}^t K(t,s)f(s)ds$$
has the same boundedness when $\beta>\alpha$, and $\|\widetilde{T}\|\lesssim\|T\|$.
\end{lem}

In particular, if $s=0$, the estimate \eqref{T} for the case $\gamma=0$ is just the same as \eqref{clastr},
and \eqref{T*} for the case $\tilde{\gamma}=0$ follows from its dual estimate.
By the $TT^*$ argument together with these cases, \eqref{TT*} follows for the case where $s=0$ and $\gamma,\tilde{\gamma}=0$.
It follows also directly from \eqref{weism} that \eqref{T} and \eqref{T*} hold for $(q,r)=(2,2)$ with $\gamma=1+s$ and $(\tilde{q},\tilde{r})=(2,2)$ with $\tilde{\gamma}=1-s$, respectively.
This shows Remark \ref{opl}.

\section{Nonlinear estimates}\label{sec3}
In this section we obtain some weighted estimates for the nonlinearity of the INLS equation
using the same spaces as those involved in the weighted Strichartz estimates in Proposition \ref{Prop1}.
These nonlinear estimates will play a crucial role in later sections to obtain the well-posedness results
applying the contraction mapping principle.

\subsection{The $L^2$ case}
We first establish the nonlinear estimates which are necessary for us to obtain the well-posedness in $L^2$
in the next section.

\begin{lem}\label{le1}
Let $d\ge 3$, $0< \alpha < 2 $ and $0 < \beta \leq (4-2\alpha)/d$.
Assume that the exponents\, $q,r,\gamma$ satisfy all the conditions given as in Theorem \ref{Thm1}.
Then there exist certain $\tilde{q},\tilde{r},\tilde{\gamma}$ satisfying all the conditions given as in Proposition \ref{Prop1} with $s=0$
for which
\begin{align}
\label{Festi}
\left\| |x|^{-\alpha} |u|^{\beta} v \right\|_{L_t^{\tilde{q}'}(I ; L_{x}^{\tilde{r}'}(|x|^{\tilde{r}' \tilde{\gamma}}))}
\leq T^{\theta_0} \left\|  u \right\|_{L_t^{q}(I; L_{x}^{r}(|x|^{-r\gamma}))} ^{\beta} \left\|  v \right\|_{L_t^{q}(I; L_{x}^{r}(|x|^{-r\gamma}))}
\end{align}
holds for any finite interval $I=[0,T]$ and
\begin{equation}\label{th1}
\theta_0 = -\frac{d\beta}{4} + 1 - \frac{\alpha}{2}.
\end{equation}
\end{lem}

\begin{rem}
It should be noted that $\theta_0\geq0$ if and only if $\beta\leq(4-2\alpha)/d$.
\end{rem}

\begin{proof}[Proof of Lemma \ref{le1}]
Let us first consider the exponent pairs $(q, r,\gamma)$ and $(\tilde{q}, \tilde{r},\widetilde{\gamma})$
satisfying $q>\tilde{q}'$,
\begin{equation}\label{r}
 \frac{2}{q} = d(\frac{1}{2} - \frac{1}{r}) +\gamma,
 \quad \frac{\gamma}{2} < \frac{1}{q} \leq \frac{1}{2},\quad\frac{\gamma}{2} \leq \frac{1}{r} <\frac{1}{2},
 \quad 0< \gamma < 1,
\end{equation}
\begin{equation}\label{r2}
 \frac{2}{\tilde{q}} = d(\frac{1}{2} - \frac{1}{\tilde{r}}) + \tilde{\gamma},
 \quad \frac{\tilde{\gamma}}{2} < \frac{1}{\tilde{q}} \leq \frac{1}{2},
 \quad\frac{\tilde{\gamma}}{2} \leq \frac{1}{\tilde{r}} < \frac{1}{2},
 \quad 0< \tilde{\gamma} < 1,
\end{equation}
which are just the same given as in Proposition \ref{Prop1} with $s=0$ for which the estimates therein hold.
We then let
\begin{equation}\label{setting}
 \frac{1}{\tilde{q}'}= \theta_0 + \frac{\beta+1}{q},\quad \frac{1}{\tilde{r}'}=\frac{\beta+1}{r} \quad \textnormal{and} \quad \tilde{\gamma}-\alpha=-\gamma(\beta+1).
\end{equation}
By combining the first conditions in \eqref{r} and \eqref{r2} together with \eqref{setting},
it is not difficult to see that $\theta_0$ is determined by
\begin{equation*}
\theta_0 = \frac{1}{\tilde{q}'} - \frac{\beta+1}{q} = -\frac{d\beta}{4} + 1 - \frac{\alpha}{2}
\end{equation*}
as in \eqref{th1}.

Next we use H\"{o}lder's inequality repeatedly along with \eqref{setting} to get
\begin{align*}
\left\| |x|^{-\alpha} |u|^{\beta} v \right\|_{L_t^{\tilde{q}'}(I ; L_{x}^{\tilde{r}'}(|x|^{\tilde{r}' \tilde{\gamma}}))}
&=\left\||x|^{\tilde{\gamma}-\alpha}|u|^{\beta} v  \right\|_{L_t^{\tilde{q}'}(I;L_x^{\tilde{r}'})} \\
& = \left\| |x|^{-\gamma(\beta+1)} |u|^{\beta} v \right\|_{L_t^{\tilde{q}'}(I ; L_{x}^{\frac{r}{\beta+1}})} \\
& \leq  T^{\theta_0} \left\|  |x|^{-\gamma(\beta+1)} |u|^{\beta} v \right\|_{L_t^{\frac{q}{\beta+1}}(I ; L_{x}^{\frac{r}{\beta+1}})} \\
& \leq T^{\theta_0} \left\|  u \right\|_{L_t^{q}(I; L_{x}^{r} (|x|^{-r\gamma}))}^{\beta} \left\| v \right\|_{L_t^{q}(I; L_{x}^{r}(|x|^{-r\gamma}))}
\end{align*}
as desired in \eqref{Festi}.

Now it remains to show the requirements on $q$, $r$ and $\gamma$ for which \eqref{Festi} holds.
We first use the last two conditions in \eqref{setting} to transform the exponents $\tilde{r}$ and $\tilde{\gamma}$ in \eqref{r2} to $r$ and $\gamma$, as follows:
\begin{equation}\label{r11}
\frac{2}{\tilde{q}} = d(\frac{\beta+1}{r}-\frac{1}{2}) + \alpha-\gamma(\beta+1),
\end{equation}
\begin{equation}\label{r111}
\frac{\alpha-\gamma(\beta+1)}{2} < \frac{1}{\tilde{q}}\leq \frac{1}{2},\quad \frac{1}{2(\beta+1)} < \frac{1}{r} \leq \frac{2-\alpha}{2(\beta+1)}+\frac{\gamma}{2},
\end{equation}
and
\begin{equation}\label{rrt}
0< \alpha-\gamma(\beta+1) < 1.
\end{equation}
Next we substitute \eqref{r11} into the first one in \eqref{r111} to get
\begin{equation}\label{r11'}
  \frac{1}{2(\beta+1)} < \frac{1}{r} \leq \frac{d-2(\alpha-1)}{2d(\beta+1)} +\frac{\gamma}{d}.
\end{equation}
Similarly we get
\begin{equation}\label{rlower}
 \frac{d-2}{2d} + \frac{\gamma}{d} \leq \frac{1}{r} < \frac{1}{2}
\end{equation}
by substituting the first one into the second one in \eqref{r}.
Notice here that the range of $1/r$ in \eqref{r} may be replaced by \eqref{rlower}
because $\frac{\gamma}{2}\leq\frac{d-2}{2d}+\frac{\gamma}{d}$ when $\gamma<1$,
and similarly using \eqref{rrt}, it is easy to see that the second one in \eqref{r111} may be replaced by \eqref{r11'}.
Finally, using the last two conditions in \eqref{setting} and the first conditions in \eqref{r} and \eqref{r2},
one can easily see that $q>\tilde{q}'$ is transformed to
$\frac1r<\frac{2-\alpha}{d\beta}+\frac{\gamma}{d}$,
but this may be replaced by \eqref{r11'}.
From the first one in \eqref{setting}, we also require $1/\tilde{q}'-1/q=\theta_0+\beta/q>0$,
but this is satisfied since $\frac{1}{q}>\frac{\gamma}{2}>0$.

Hence all the conditions on $q$, $r$ and $\gamma$ which we require are summarized as follows:
\begin{equation*}
\frac{2}{q} = d(\frac{1}{2} - \frac{1}{r}) + \gamma,\quad \frac{\gamma}{2} < \frac{1}{q} \leq \frac{1}{2},
\end{equation*}
\begin{equation*}
\frac{1}{2(\beta+1)}<\frac{1}{r}
<\frac{d-2(\alpha-1)}{2d(\beta+1)}+\frac{\gamma}{d},
\end{equation*}
and
\begin{equation*}
\max \{0,\ \frac{\alpha-1}{\beta+1} \} < \gamma < \min \{ 1, \frac{\alpha}{\beta+1} \},
\end{equation*}
which are the same as in Theorem \ref{Thm1}.
\end{proof}

\subsection{The $H^s$ case}
Now we obtain some weighted estimates for the nonlinear term which will be used for the well-posedness in $H^s$ in Section \ref{sec5}.
In comparison with the previous lemma, it is more delicate to obtain Lemma \ref{le2} below in which
we need to get an additional estimate \eqref{F2} having the term $|\nabla|^{-s}$ that is quite difficult to handle in the weighted spaces.
For this, we will make use of a weighted version of the Sobolev embedding.

\begin{lem}\label{le2}
Let $d \geq 3$ and $0<s<1/3$. Assume that
\begin{equation}\label{102}
\max\{\frac{26-3d}{12},\frac{12s+4ds-8s^2}{d+4s}\}<\alpha<2
\end{equation}
and
\begin{equation}\label{1012}
\max\{0,\frac{10s-2\alpha}{d-6s}\}< \beta \leq \frac{4-2\alpha}{d-2s}.
\end{equation}
If the exponents $q,r,\gamma$ satisfy all the conditions given as in Theorem \ref{Thm2},
then there exist certain $\tilde{q}_i,\tilde{r}_i,\tilde{\gamma}_i$, $i=1,2$, satisfying all the conditions given as in Proposition \ref{Prop1} with $s >0$ for which
\begin{equation}\label{F1}
\left \||x|^{-\alpha} |u|^{\beta} v \right\|_{L_t^{\tilde{q}_1'}(I;L_x^{\tilde{r}_1'}(|x|^{\tilde{r}_1'\tilde{\gamma}_1}))}\lesssim T^{\theta_1} \|u\|_{L_t^q(I;L_x^r(|x|^{-r \gamma}))}^{\beta}\|v\|_{L_t^q(I;L_x^r(|x|^{-r \gamma}))}
\end{equation}
and
\begin{align}\label{F2}
\nonumber\big\| |\nabla|^{-s} (|x|^{-\alpha} |u|^{\beta} v)&\big\|_{L_t^{\tilde{q}_2'}(I;L_x^{\tilde{r}_2'}(|x|^{\tilde{r}_2'\tilde{\gamma}_2}))}\\
 &\lesssim T^{\theta_2} \|u\|_{L_t^q(I;L_x^r(|x|^{-r \gamma}))}^{\beta} \|v\|_{L_t^q(I;L_x^r(|x|^{-r \gamma}))}
\end{align}
hold for any finite interval $I=[0, T]$ and
\begin{equation}\label{th2}
\theta_1= -\frac{d\beta}{4}+1-\frac{\alpha}{2}+\frac{s\beta}{2},\quad  \quad \theta_2=-\frac{d\beta}{4} +1-\frac{\alpha}{2} +\frac{s(\beta+1)}{2}.
\end{equation}
\end{lem}

\begin{rem}
We note that $\beta \leq (4-2\alpha)/(d-2s)$ if and only if $\theta_1 \ge 0 $ (or $\theta_2 \ge s/2$).
\end{rem}

\begin{proof}[Proof of Lemma \ref{le2}]
First we consider the exponent pairs $(q,r,\gamma)$ and $(\tilde{q}_i, \tilde{r}_i, \tilde{\gamma}_i)$ for $i=1,2$ satisfying $q>\tilde{q}_i'$,
\begin{equation}\label{Hs2}
\frac{2}{q}=d ( \frac{1}{2}- \frac{1}{r} )+\gamma-s,\quad
\frac{\gamma-s}{2}< \frac{1}{q} \leq \frac{1}{2},\quad
\frac{\gamma-s}{2}\leq \frac{1}{r} < \frac{1}{2},
\quad 3s < \gamma < 1+s,
\end{equation}
\begin{equation}\label{Hs3}
\frac{2}{\tilde{q}_i} = d ( \frac{1}{2}- \frac{1}{\tilde{r}_i} )  + \tilde{\gamma}_i+s,
\ \frac{\tilde{\gamma}_i+s}{2}< \frac{1}{\tilde{q}_i} \leq \frac{1}{2},
\ \frac{\tilde{\gamma}_i+s}{2}\leq  \frac{1}{\tilde{r}_i} < \frac{1}{2},
\ s < \tilde{\gamma}_i < 1-s,
\end{equation}
which are just the same given as in Proposition \ref{Prop1} with $0<s<1/2$ for which the estimates therein hold.

\subsubsection{Proof of \eqref{F1}}
The first estimate \eqref{F1} is obtained in a similar way as in the previous subsection.
Let us first set
\begin{equation}\label{set2}
\frac{1}{{\tilde{q}_1}'}=\theta_1+\frac{\beta+1}{q},\quad \frac{1}{{\tilde{r}_1}'}=\frac{\beta+1}{r} \quad \textnormal{and} \quad \tilde{\gamma}_1 -\alpha = -\gamma(\beta+1).
\end{equation}
By combining the first conditions in \eqref{Hs2} and \eqref{Hs3} for $i=1$ together with \eqref{set2},
it is then easy to see that $\theta_1$ is determined by
\begin{equation*}
\theta_1=\frac{1}{{\tilde{q}_1}'}-\frac{\beta+1}{q}=-\frac{d\beta}{4}+1-\frac{\alpha}{2}+\frac{s\beta}{2}
\end{equation*}
as in \eqref{th2}.

Next we apply H\"{o}lder's inequality repeatedly along with \eqref{set2} to get
\begin{equation*}	
\begin{aligned}
\left \||x|^{-\alpha} |u|^{\beta} v \right\|_{L_t^{\tilde{q}_1'}(I;L_x^{\tilde{r}_1'}(|x|^{\tilde{r}_1'\tilde{\gamma}_1}))}
&= \left\||x|^{\tilde{\gamma}_1-\alpha}|u|^{\beta} v  \right\|_{L_t^{\tilde{q}_1'}(I;L_x^{\tilde{r}_1'})} \cr
&= \left\||x|^{ -\gamma(\beta+1)} |u|^{\beta} v  \right\|_{L_t^{\tilde{q}_1'}(I;L_x^{\frac{r}{\beta+1}})} \cr
& \leq T^{\theta_1} \left\||x|^{ -\gamma(\beta+1)}|u|^{\beta} v  \right\|_{L_t^{\frac{q}{\beta+1}}(I;L_x^{\frac{r}{\beta+1}})} \cr
& \leq T^{\theta_1} \left\|u  \right\|_{L_t^{q}(I;L_x^{r}(|x|^{-r\gamma}))}^{\beta}  \left\| v \right\|_{L_t^{q}(I;L_x^{r}(|x|^{-r\gamma}))}
	\end{aligned}
\end{equation*}
as desired in \eqref{F1}.

Now we only need to show the requirements on $q$, $r$ and $\gamma$ for which \eqref{F1} holds.
 Using the last two conditions in \eqref{set2}, we first change the exponents $\tilde{r}_1$ and $\tilde{\gamma}_1$ in \eqref{Hs3} for $i=1$ into $r$ and $\gamma$, as follows:
\begin{gather}
\label{q1} \frac{2}{\tilde{q}_1} = -\frac{d}{2}+ \frac{d(\beta+1)}{r}+\alpha+s-\gamma(\beta+1), \\
\label{q2} \frac{\alpha-\gamma(\beta+1)+s}{2}< \frac{1}{\tilde{q}}_1 \leq \frac{1}{2},
\quad \frac{1}{2(\beta+1)}<  \frac{1}{r} \leq \frac{2-\alpha-s}{2(\beta+1)}+\frac{\gamma}{2},
\end{gather}
\begin{gather}\label{g1}
\frac{\alpha+s-1}{\beta+1}< \gamma <\frac{\alpha-s}{\beta+1}.
\end{gather}
Next we insert \eqref{q1} into the first one in \eqref{q2} to get
\begin{equation}\label{rrr}
\frac{1}{2(\beta+1)}< \frac{1}{r} \leq  \frac{1 - \alpha - s}{d(\beta+1)}+\frac{1}{2(\beta+1)}+\frac{\gamma}{d}
\end{equation}
from which the second one in \eqref{q2} can be removed using \eqref{g1}.
Finally, all the requirements deduced from $q>\tilde{q}_1'$ are already satisfied by the other requirements
similarly as before.

In summary, the requirements on $q$, $r$ and $\gamma$ for which \eqref{F1} holds are \eqref{Hs2}, \eqref{g1} and \eqref{rrr}
which are less restrictive than those in Theorem \ref{Thm2}.
But we will show that the common requirements for which both \eqref{F1} and \eqref{F2} hold are given exactly by
those in Theorem \ref{Thm2}.

\subsubsection{Proof of \eqref{F2}}
Now we have to obtain \eqref{F2} under the requirements \eqref{Hs2}, \eqref{g1} and \eqref{rrr} on $q$, $r$ and $\gamma$
for which \eqref{F1} holds.
Let us first set
\begin{equation}\label{H1}
\frac{1}{\tilde{q}_2'}=\theta_2+\frac{\beta+1}{q}
\end{equation}
and then use H\"older's inequality in $t$ to the left-hand side of \eqref{F2} to get
\begin{align*}
\left\|\, |\nabla|^{-s} (|x|^{-\alpha}|u|^{\beta}v) \right\|_{L_{t}^{\tilde{q}_2'}(I;L_x^{\tilde{r}_2'}(|x|^{\tilde{r}_2'\tilde{\gamma}_2}))}
&=\left\|\,|x|^{\tilde{\gamma}_2} |\nabla|^{-s} (|x|^{-\alpha}|u|^{\beta}v) \right\|_{L_{t}^{\tilde{q}_2'}(I;L_x^{\tilde{r}_2'})} \cr
&\leq T^{\theta_2} \left\|\,|x|^{\tilde{\gamma}_2} |\nabla|^{-s} (|x|^{-\alpha}|u|^{\beta}v)\right\|_{L_{t}^{\frac{q}{\beta+1}}(I;L_x^{{\tilde{r}_2}'})}.
\end{align*}
To handle the term $|\nabla|^{-s}$ here, we make use of the following lemma which is a weighted version of the Sobolev embedding.

\begin{lem}[\cite{SW}] \label{le3}
Let $d\ge 1$ and $0<s<d$. If
\begin{equation*}
1<p\leq q<\infty,\quad -d/q<b\leq a<d/p' \quad\text{and}\quad a-b-s=d/q-d/p,
\end{equation*}
then
\begin{equation}\label{soso}
\||x|^{b}f\|_{L^q} \leq C_{a, b, p, q} \||x|^{a} |\nabla|^s f\|_{L^p}.
\end{equation}
\end{lem}

Indeed, applying \eqref{soso} with $a=\alpha-\gamma(\beta+1)$, $b=\tilde{\gamma}_2$, $p=\frac{r}{\beta+1}$, $q=\tilde{r}_2'$, and $f=|x|^{-\alpha}|u|^{\beta}v$,
and then using H\"older's inequality, we get for $0<s<d$
\begin{align*}
 \left\|\,|x|^{\tilde{\gamma}_2} |\nabla|^{-s} (|x|^{-\alpha}|u|^{\beta}v) \right\|_{L_{t}^{\frac{q}{\beta+1}}(I;L_x^{\tilde{r}_2'})}
&\lesssim  \left\|  |x|^{-\gamma(\beta+1)}|u|^{\beta}v \right\|_{L_{t}^{\frac{q}{\beta+1}}(I;L_x^{\frac{r}{\beta+1}})} \cr
&\lesssim  \left\|  |x|^{-\gamma} u\right\|_{L_{t}^{q}(I;L_x^{r})} ^{\beta} \left\|  |x|^{-\gamma} v\right\|_{L_{t}^{q}(I;L_x^{r})}
\end{align*}
if
\begin{equation}\label{H3}
 1< \frac{r}{\beta+1} \leq \tilde{r}_2' < \infty,
\end{equation}
\begin{equation}\label{H5}
-\frac{d}{\tilde{r}_2'}<\tilde{\gamma}_2\leq\alpha-\gamma(\beta+1)<d-\frac{d(\beta+1)}{r},
\end{equation}
and
\begin{equation}\label{H4}
 \tilde{\gamma}_2= \alpha-\gamma(\beta+1) -s -\frac{d}{\tilde{r}_2'}+\frac{d(\beta+1)}{r}.
\end{equation}
Since $\tilde{\gamma}_2>0$, the first inequality in \eqref{H5} is redundant,
and since the upper bound of $1/r$ in \eqref{rrr} is less than the one which follows from the third inequality in \eqref{H5},
it is also redundant. Hence \eqref{H5} is reduced to
\begin{equation}\label{H55}
\tilde{\gamma}_2\leq\alpha-\gamma(\beta+1).
\end{equation}

Now we replace $\tilde{\gamma}_2$ in the first condition of \eqref{Hs3} for $i=2$ with \eqref{H4} to get
\begin{equation}\label{q2tilde}
\frac{2}{\tilde{q}_2}= -\frac{d}{2}+\alpha-\gamma(\beta+1)+\frac{d(\beta+1)}{r}.
\end{equation}
Then by inserting \eqref{q2tilde} and the first one of \eqref{Hs2} into \eqref{H1}, we see that
$$\theta_2 =\frac{1}{\tilde{q}_2'}-\frac{\beta+1}{q}=-\frac{d\beta}{4}+1-\frac{\alpha}{2}+\frac{s(\beta+1)}{2},$$
as in \eqref{th2}.

By using \eqref{H4} and \eqref{q2tilde}, the exponents $\tilde{q}_2$ and $\tilde{\gamma}_2$ in all the inequalities in \eqref{Hs3} for $i=2$
can be removed as follows:
\begin{equation}\label{H11}
\frac{1}{r}\leq \frac{1-\alpha}{d(\beta+1)}+\frac{1}{2(\beta+1)}+\frac{\gamma}{d},\quad \frac{1}{\tilde{r}_2} < \frac12,
\end{equation}
\begin{equation}\label{H12}
\alpha-\gamma(\beta+1) -\frac{d}{\tilde{r}_2'}+\frac{d(\beta+1)}{r} \leq \frac{2}{\tilde{r}_2}<1,
\end{equation}
\begin{equation}\label{H22}
2s < \alpha-\gamma(\beta+1) -\frac{d}{\tilde{r}_2'}+\frac{d(\beta+1)}{r} < 1.
\end{equation}
Here, \eqref{H11} and \eqref{H12} come from the second and third ones in \eqref{Hs3}, respectively,
while \eqref{H22} is derived from the last one in \eqref{Hs3}.
Similarly, \eqref{H3} and \eqref{H55} are transformed to
\begin{equation}\label{H33}
0\leq \frac{\beta+1}{r}-\frac{1}{\tilde{r}_2'} <\frac{1}{\tilde{r}_2}\quad\text{and}\quad \frac{d(\beta+1)}{r}-\frac{d}{\tilde{r}_2'} \leq s,
\end{equation}
respectively.
It is immediate that the first inequality in \eqref{H11} is redundant by \eqref{rrr}.
Since $\tilde{r}_2 > 2$, \eqref{H12} and \eqref{H22} are reduced to
\begin{equation}\label{rt0}
2s < \alpha-\gamma(\beta+1) -\frac{d}{\tilde{r}_2'}+\frac{d(\beta+1)}{r} \leq \frac{2}{\tilde{r}_2}
\end{equation}
which is possible for $s<1/\tilde{r}_2$.
If $s<1/\tilde{r}_2$, \eqref{H33} is then reduced to
\begin{equation}\label{rt1}
0 \leq \frac{\beta+1}{r}-\frac{1}{{\tilde{r}_2'}} \leq \frac{s}{d}.
\end{equation}
In summary, we are reduced to \eqref{rt0} and  \eqref{rt1}, which are equivalent to
\begin{equation}\label{rt2}
\frac{2s-\alpha+\gamma(\beta+1)}{d}-\frac{\beta+1}{r}+1 < \frac{1}{\tilde{r}_2} \leq \frac{d-\alpha+\gamma(\beta+1)}{d-2}-\frac{d(\beta+1)}{(d-2)r}
\end{equation}
and
\begin{equation}\label{rt3}
1-\frac{\beta+1}{r}\leq \frac{1}{\tilde{r}_2} \leq \frac{s}{d} + 1 -\frac{\beta+1}{r}
\end{equation}
respectively, under the condition
\begin{equation}\label{rt45}
s<\frac1{\tilde{r}_2}<\frac12.
\end{equation}

Now, it suffices to check that there exists $\tilde{r}_2$ satisfying \eqref{rt2}, \eqref{rt3} and \eqref{rt45} under the conditions \eqref{Hs2}, \eqref{g1} and \eqref{rrr}.
To do so, we first make each lower bound of $1/\tilde{r}_2$ in \eqref{rt2}, \eqref{rt3} and \eqref{rt45} less than all the upper bounds in turn.
Then we are reduced to
\begin{equation}\label{r22}
\frac{1}{r}< \frac{1-s}{\beta+1} + \frac{2s-\alpha}{d(\beta+1)} + \frac{\gamma}{d}
\end{equation}
and
\begin{equation}\label{rrr2}
\frac{1}{2(\beta+1)}+\frac{2s-\alpha}{d(\beta+1)}+\frac{\gamma}{d} <\frac{1}{r}.
\end{equation}
Indeed, starting the process from the lower bound in \eqref{rt2}, we arrive at \eqref{r22}, $s-\alpha+\gamma(\beta+1)<0$, and \eqref{rrr2},
but the second one is redundant by \eqref{g1}.
Similarly from the lower bound in \eqref{rt3}, we arrive at $\frac1r\leq\frac{2-\alpha}{2(\beta+1)}+\frac\gamma2$, $0\leq \frac sd$, and
$\frac{1}{2(\beta+1)}<\frac1r$. But here the first and third ones are replaced by \eqref{rrr} under \eqref{g1},
and the second one is redundant.
Finally from the lower bound in \eqref{rt45}, we arrive at \eqref{r22}, $\frac{1}{r}< \frac{1-s}{\beta+1}+\frac{s}{d(\beta+1)}$, and
$s<1/2$, but the second one is replaced by \eqref{r22} using \eqref{g1} and the last one is redundant.
Nextly, we note that \eqref{r22} may be also eliminated by \eqref{rrr} using the fact that $d-2-2sd+6s =(d-2)(1-2s)+2s>0$,
and we combine the first two conditions in \eqref{Hs2} to conclude
\begin{equation}\label{rt4}
\frac{1}{2}-\frac{1-\gamma+s}{d} \leq \frac{1}{r} < \frac{1}{2},
\end{equation}
which replaces the third condition in \eqref{Hs2} since $\gamma-s<1$.
Finally, all the requirements deduced from $q>\tilde{q}_2'$ are already satisfied by the other requirements
similarly as before.

Hence, all the requirements on $r$ are \eqref{rrr}, \eqref{rt4} and \eqref{rrr2} for which both \eqref{F1} and \eqref{F2} hold.
In summary,
\begin{equation}\label{r3}
\max \lbrace \frac{1}{2(\beta+1)},\ \frac{d-2s-2}{2d}+\frac{\gamma}{d} \rbrace <  \frac{1}{r}
 < \min \lbrace   \frac{d-2s-2(\alpha-1)}{2d(\beta+1)} + \frac{\gamma}{d}, \frac{1}{2} \rbrace
\end{equation}
and
\begin{equation}\label{r4}
\frac{1}{2(\beta+1)} + \frac{2s-\alpha}{d(\beta+1)} + \frac{\gamma}{d} < \frac{1}{r},
\end{equation}
which are the same as in Theorem \ref{Thm2}.
Here, omitting the equality in \eqref{rt4} is harmless.

Next we check that there exist $r$ satisfying \eqref{r3} and \eqref{r4}.
For this, we make each lower bound of $1/r$ in them less than all the upper bounds in \eqref{r3} in turn.
Then the restrictions on $\gamma$ and $s$ are deduced as
\begin{equation}\label{g2}
\quad \gamma<\frac{d\beta+2\alpha-4s}{2(\beta+1)}
\end{equation}
and $0<s<1/3$.
In fact, staring the process from the first lower bound in \eqref{r3}, we arrive at $\alpha + s -1 -\gamma(\beta+1) < 0$ and $\beta > 0$,
but the former is redundant by \eqref{g1} as well as $\beta>0$.
Similarly from the second lower bound in \eqref{r3}, we arrive at $\beta < \frac{4-2\alpha}{d-2-2s}$ and $\gamma < 1+s$
which are obviously redundant.
Finally from the lower bound in \eqref{r4}, we arrive at $s<1/3$ and \eqref{g2}.

Hence all the requirements on $\gamma$ for which both \eqref{F1} and \eqref{F2} hold are $3s< \gamma < 1+s$, \eqref{g1} and \eqref{g2}.
In summary,
\begin{equation}\label{g4}
\max \{3s,\, \frac{\alpha+s-1}{\beta+1}\} < \gamma < \min \{1+s,\, \frac{\alpha-s}{\beta+1},\ \frac{d\beta+2\alpha-4s}{2(\beta+1)} \}
\end{equation}
as in Theorem \ref{Thm2}.

To guarantee $\gamma$ satisfying all the requirements in \eqref{g4} under $0<s<1/3$,
we make each lower bound of $\gamma$ in \eqref{g4} less than all the upper bounds in turn.
Then we are reduced to
\begin{equation}\label{beta1}
\beta < \frac{\alpha-4s}{3s} \quad \textnormal{and} \quad \frac{10s-2\alpha}{d-6s} < \beta.
\end{equation}
Indeed, starting from the first lower bound, we see $s<1/2$ and \eqref{beta1}.
Nextly, from the second lower bound, we have $\alpha-2 < \beta + s\beta$, $s<1/2$ and $6s-2 < d\beta$
which are clearly redundant.

Now the first condition in \eqref{beta1} is satisfied if $\frac{4-2\alpha}{d-2s} < \frac{\alpha-4s}{3s}$,
which is equivalent to
\begin{equation*}
 \frac{12s+4ds-8s^2}{d+4s} < \alpha
\end{equation*}
as in Theorem \ref{Thm2}.
Then we need to check that there exist $\alpha \in (\frac{12s+4ds-8s^2}{d+4s}, 2)$.
This is possible if
\begin{equation}\label{funcs}
-4s^2+2s+2ds<d
\end{equation}
holds for $0<s<1/3$.
To show this, we consider a quadratic function $f_1(s) = -4s^2 +2(1+d)s$ which is concave downward.
Since $f_1(s)$ is an increasing function for $s \in (0,\frac{d+1}{4})$ including $s \in (0, \frac{1}{3})$,
the inequality \eqref{funcs} follows clearly from the fact that $f_1(\frac{1}{3})<d$.
Next we check that there exist $\beta$ satisfying the second condition of \eqref{beta1} under $0<s<1/3$.
For this, we need to show
$ \frac{10s-2\alpha}{d-6s} < \frac{4-2\alpha}{d-2s} $
which is equivalent to
\begin{equation}\label{mcv}
-10s^2+(5d+12-4\alpha)s < 2d.
\end{equation}
Similarly, we consider a quadratic function $f_2(s)=-10s^2+(5d+12-4\alpha)s$.
It is also concave downward and increases for $s < \frac{5d+12-4\alpha}{20}$.
Since $\frac{1}{3} < \frac{5d+12-4\alpha}{20}$,
we require $f_2(\frac{1}{3}) < 2d$ to show \eqref{mcv}, which is equivalent to
$$\frac{26-3d}{12} < \alpha$$
as in Theorem \ref{Thm2}.
Hence, $\alpha$ and $\beta$ are required as in \eqref{102} and \eqref{1012}.
Since $\frac{d}{2} - \frac{2-\alpha}{\beta}\leq s$ (i.e., $\beta \leq \frac{4-2\alpha}{d-2s}$),
we finally require
\begin{equation*}
\frac{d}{2} - \frac{2-\alpha}{\beta} <\frac13
\end{equation*}
which is equivalent to $\beta < \frac{12-6\alpha}{3d-2}$.
But here, $\frac{4-2\alpha}{d-2s} < \frac{12-6\alpha}{3d-2}$ for $0<s<1/3$.
No more requirement occurs.
This completes the proof.
\end{proof}

\section{The well-posedness in $L^2$}\label{sec4}
In this section we prove Theorem \ref{Thm1} and Corollary \ref{cor} by applying the contraction mapping principle
along with the weighted Strichartz estimates in Proposition \ref{Prop1} with $s=0$.
The nonlinear estimates in Lemma \ref{le1} play a key role in this step.

\subsection{The subcritical case}

By Duhamel's principle, we first write the solution of the Cauchy problem \eqref{INLS} as
\begin{equation}\label{Duhamel}
\Phi_{u_0}(u) = e^{it \Delta} u_0 - i\lambda  \int_{0}^{t} e^{i(t-\tau)\Delta}F(u) d\tau
\end{equation}
where $F(u)=|\cdot|^{-\alpha} |u(\cdot,\tau)|^{\beta}u(\cdot, \tau)$.
For appropriate values of $T,M>0$, we shall show that $\Phi$ defines a contraction map on
\begin{align*}\label{contset}
X(T,M)=\Big\lbrace u\in  &C_t(I;L_x^2)\cap L_t^q(I;L_x^r(|x|^{-r\gamma})) :\\
&\sup_{t\in I}\|u\|_{L_x^2}+\|u\|_{L_t^q(I;L_x^r(|x|^{-r\gamma}))} \leq  M \Big\rbrace,
\end{align*}
equipped with the distance
$$d(u,v)=\sup_{t\in I}\|u-v\|_{L_x^2}+\|u-v\|_{L_t^q(I;L_x^r(|x|^{-r\gamma}))}.$$
Here, $I=[0,T]$ and $(q,r,\gamma)$ is given as in Theorem \ref{Thm1}.

For this, we first show that $\Phi$ is well defined on $X$. Namely, for $u\in X$
\begin{equation}\label{fgj}
\sup_{t\in I}\|\Phi(u)\|_{L_x^2}+\|\Phi(u)\|_{L_t^q(I;L_x^r(|x|^{-r\gamma}))} \leq M.
\end{equation}
Using \eqref{T} and \eqref{TT*} combined with \eqref{Duhamel}, we obtain
\begin{equation*}
  \| \Phi(u)\|_{L_t^{q}(I ; L_{x}^{r}(|x|^{-r \gamma}))}
  \leq C\| u_0 \|_{L^2} + C\left\||x|^{-\alpha}|u|^{\beta}u \right\|_{L_{t}^{\tilde{q}'} (I; L_{x}^{\tilde{r}'}(|x|^{\tilde{r}' \tilde{\gamma}}))}.
\end{equation*}
By applying Lemma \ref{le1}, it follows then that
\begin{align}
\nonumber\left\| \Phi(u) \right\|_{L_{t}^{q}(I; L_{x}^{r}(|x|^{-r\gamma}))}
&\leq C\| u_0 \|_{L^2} +  CT^{\theta_0} \| u \|_{L_t^{q}(I; L_{x}^{r} (|x|^{-r\gamma}))} ^{\beta+1} \\
\label{1j}&\leq C\| u_0 \|_{L^2} +  CT^{\theta_0} M^{\beta+1}.
\end{align}
On the other hand, applying Plancherel's theorem, \eqref{T*} and \eqref{Festi} in turn, we see
\begin{align}\label{conv}
\nonumber\sup_{t\in I}\|\Phi(u)\|_{L_x^2} &\leq C\|u_0\|_{L^2} +C\bigg\|\int_{-\infty}^{\infty} e^{-i\tau\Delta}\chi_{[0,t]}(\tau)F(u) d\tau\bigg\|_{L_x^2}\\
\nonumber&\leq C\|u_0\|_{L^2}+C\|F(u)\|_{L_t^{\tilde{q}'}(I; L_{x}^{\tilde{r}'} (|x|^{\tilde{r}' \widetilde{\gamma}}))}\\
&\leq C\|u_0\|_{L^2}+CT^{\theta_0} M^{\beta+1}.
\end{align}
Hence, if we fix $M=4C\| u_0 \|_{L^2}$ and take $T>0$ such that
\begin{equation}\label{dgh0}
 CT^{\theta_0}M^{\beta}\leq\frac 18,
\end{equation}
we get \eqref{fgj}.
In the subcritical case where $\beta<(4-2\alpha)/d$, $T$ carries a positive power $\theta_0$ here.
Thus one can give a precise estimate for the life span of the solution according to the size of the initial data, $T\sim\|u_0\|_{L^2}^{-\beta/\theta_0}$.

Next we show that $\Phi$ is a contraction.
Namely, for $u,v\in X$
 \begin{equation}\label{alk}
 d(\Phi(u),\Phi(v)) \leq \frac{1}{2} d(u,v ).
 \end{equation}
Using the same arguments used in \eqref{conv}, we see
\begin{equation}\label{cont24}
\sup_{t\in I}\|\Phi(u)-\Phi(v)\|_{L_x^2}
\leq C\|F(u)-F(v)\|_{L_t^{\tilde{q}'}(I; L_{x}^{\tilde{r}'} (|x|^{\tilde{r}' \widetilde{\gamma}}))}.
\end{equation}
Then we will show
\begin{equation}\label{cont3}
 \| F(u)-F(v) \|_{L_t^{\tilde{q}'}(I; L_{x}^{\tilde{r}'} (|x|^{\tilde{r}' \tilde{\gamma}}))} \leq \frac{1}{4C} \| u-v \|_{L_t^{q}(I; L_{x}^{r}(|x|^{-r \gamma}))}.
 \end{equation}
Indeed, using the following simple inequality
\begin{equation}\label{sin}
  \left||u|^{\beta}u-|v|^{\beta}v \right| \leq C \big(|u|^{\beta}|u-v|+|v|^{\beta}|u-v|\big),
\end{equation}
we are reduced to showing
 \begin{equation}\label{cont}
 \left\| |x|^{-\alpha} |u|^{\beta} |u-v| \right\|_{L_t^{\tilde{q}'}(I; L_{x}^{\tilde{r}'} (|x|^{\tilde{r}' \tilde{\gamma}}))} \leq \frac{1}{8C} \left\| u-v \right\|_{L_t^{q}(I; L_{x}^{r}(|x|^{-r \gamma}))}
 \end{equation}
by symmetry.
For this we apply Lemma \ref{le1} with $v$ replaced by $|u-v|$ so that
\begin{align*}
\left\| |x|^{-\alpha}  |u|^{\beta}|u-v| \right\|_{L_t^{\tilde{q}'}(I; L_{x}^{\tilde{r}'} (|x|^{\tilde{r}' \tilde{\gamma}}))}
& \leq T^{\theta_0} \| u \|_{L_t^{q}(I; L_{x}^{r}(|x|^{-r \gamma}))}^{\beta}  \| u-v \|_{L_t^{q}(I; L_{x}^{r} (|x|^{-r \gamma}))} \cr
&  \leq T^{\theta_0} M^{\beta}  \| u-v \|_{L_t^{q}(I; L_{x}^{r}(|x|^{-r \gamma}))},
\end{align*}
which implies \eqref{cont} because we have chosen $T,M>0$ in \eqref{dgh0} so that $T^{\theta_0}M^{\beta}\leq1/(8C)$.
On the other hand,
\begin{equation*}
\|\Phi(u)-\Phi(v)\|_{L_{t}^{q}(I; L_{x}^{r}(|x|^{-r \gamma}))}
\leq C\|F(u)-F(v)\|_{L_t^{\tilde{q}'}(I; L_{x}^{\tilde{r}'} (|x|^{\tilde{r}' \tilde{\gamma}}))}.
	\end{equation*}
Now we obtain \eqref{alk} combining this, \eqref{cont24} and \eqref{cont3}.

Therefore, we have proved that there exists a unique local solution
$$u \in C(I ; L^{2}) \cap L^{q}(I ; L^{r}(|x|^{-r \gamma}))$$
with $T\sim\|u_0\|_{L^2}^{-\beta/\theta_0}$.
The continuous dependence of the solution $u$ with respect to the initial data $u_0$ follows clearly in the same way:
\begin{align*}
d(u,v)&\lesssim d\big(e^{it \Delta}u_0, e^{it\Delta} v_0\big) +d\bigg(\int_{0}^{t} e^{i(t-\tau)\Delta}F(u)d\tau, \int_{0}^{t} e^{i(t-\tau)\Delta}F(v) d\tau\bigg)\\
&\lesssim\|u_0-v_0\|_{L^2}.
\end{align*}
Here, $u,v$ are the corresponding solutions for initial data $u_0,v_0$, respectively.
Thanks to the mass conservation \eqref{mass},
the above process can be also iterated on translated time intervals, preserving the length of the time interval comparable to $\|u_0\|_{L^2}^{-\beta/\theta_0}$ to extend the above local solution globally in time.
The proof is now complete.

\subsection{The critical case}
The critical case requires slightly different arguments
and it yields different conclusions.
This is due to the fact that the power $\theta_0$ in the above argument becomes zero in this case.
This time we cannot gain a small power of $T$ and the smallness must have a different source.
For this reason, the global result will follow from the smallness of the initial data.
We give the main lines of the proof.

We start from showing that $\Phi$ defines a contraction on
\begin{align*}\label{contset}
\widetilde{X}(T,M,N)=\Big\lbrace u\in  &C_t(I;L_x^2)\cap L_t^q(I;L_x^r(|x|^{-r\gamma})) :\\
&\sup_{t\in I}\|u\|_{L_x^2}\leq N,\quad \|u\|_{L_t^q(I;L_x^r(|x|^{-r\gamma}))} \leq M \Big\rbrace
\end{align*}
equipped with the distance
$$d(u,v)=\sup_{t\in I}\|u-v\|_{L_x^2}+\|u-v\|_{L_t^q(I;L_x^r(|x|^{-r\gamma}))}.$$
First, we see as in \eqref{conv} and \eqref{1j} that
\begin{equation*}
\sup_{t\in I}\|\Phi(u)\|_{L_x^2}\leq C\|u_0\|_{L^2}+C M^{\beta+1}
\end{equation*}
and
$$\|\Phi(u)\|_{L_t^q(I;L_x^r(|x|^{-r\gamma}))}\leq \|e^{it\Delta}u_0\|_{L_t^q(I;L_x^r(|x|^{-r\gamma}))}+C M^{\beta+1},$$
respectively.
Observe here that
$$\|e^{it\Delta}u_0\|_{L_t^q(I;L_x^r(|x|^{-r\gamma}))}\leq \varepsilon$$
for some $\varepsilon>0$ small enough which will be chosen later, provided that either $\|u_0\|_{L^2}$ is small (see \eqref{T} with $s=0$)
or it is satisfied for some $T>0$ small enough by the dominated convergence theorem.
Hence, one can take $T=\infty$ in the first case and $T$ to be this small time in the second.
We therefore get $\Phi(u)\in \widetilde{X}$ for $u\in\widetilde{X}$ if
\begin{equation}\label{az}
C\|u_0\|_{L^2}+C M^{\beta+1}\leq N\quad\text{and}\quad \varepsilon+C M^{\beta+1}\leq M.
\end{equation}
On the other hand, using the same argument employed to show \eqref{alk}, we see
\begin{align*}
d(\Phi(u),\Phi(v))
&=\sup_{t\in I}\|\Phi(u)-\Phi(v)\|_{L_x^2}+\|\Phi(u)-\Phi(v)\|_{L_t^q(I;L_x^r(|x|^{-r\gamma}))}\\
&\leq CM^\beta\| u-v \|_{L_t^{q}(I; L_{x}^{r}(|x|^{-r \gamma}))}\\
&\leq  CM^\beta d(u,v).
\end{align*}
Now by setting $N=2C\|u_0\|_{L^2}$ and $M=2\varepsilon$ and then choosing $\varepsilon>0$ small enough so that
\eqref{az} holds and $CM^\beta \leq1/2$,
it follows that $\widetilde{X}$ is stable by $\Phi$ and $\Phi$ is a contraction on $\widetilde{X}$.

Therefore, there exists a unique local solution $u \in C(I ; L^{2}) \cap L^{q}(I ; L^{r}(|x|^{-r \gamma}))$
in the time interval $[0,T]$ with a small $T$.
Recall from the above argument that when $\|u_0\|_{L^2}$ is small enough, we can take $T=\infty$ to obtain a global solution.
The continuous dependence on the initial data $u_0$ follows clearly in the same way as before.
It remains to prove the scattering property. Following the argument above, one can easily see that
\begin{align*}
\big\|e^{-it_2\Delta}u(t_2)-e^{-it_1\Delta}u(t_1)\big\|_{L_x^2}&=\bigg\|\int_{t_1}^{t_2}e^{-i\tau\Delta}F(u)d\tau\bigg\|_{L_x^2}\\
&\lesssim\|F(u)\|_{L_t^{\tilde{q}'}([t_1,t_2]; L_{x}^{\tilde{r}'} (|x|^{\tilde{r}' \tilde{\gamma}}))}\\
&\lesssim\| u \|_{L_t^{q}([t_1,t_2]; L_{x}^{r} (|x|^{-r\gamma}))} ^{\beta+1} \quad\rightarrow\quad0
\end{align*}
as $t_1,t_2\rightarrow\infty$. This implies that
$$\varphi:=\lim_{t\rightarrow\infty}e^{-it\Delta}u(t)$$
exists in $L^2$. Furthermore, one has
$$u(t)-e^{it\Delta}\varphi= i\lambda  \int_{t}^{\infty} e^{i(t-\tau)\Delta}F(u) d\tau,$$
and hence
\begin{align*}
\big\|u(t)-e^{it\Delta}\varphi\big\|_{L_x^2}=\bigg\|\int_{t}^{\infty} e^{i(t-\tau)\Delta}F(u) d\tau\bigg\|_{L_x^2}
&\lesssim\|F(u)\|_{L_t^{\tilde{q}'}([t,\infty); L_{x}^{\tilde{r}'} (|x|^{\tilde{r}' \tilde{\gamma}}))}\\
&\lesssim\| u \|_{L_t^{q}([t,\infty); L_{x}^{r} (|x|^{-r\gamma}))} ^{\beta+1} \quad\rightarrow\quad0
\end{align*}
as $t\rightarrow\infty$.
This completes the proof.

\section{The well-posedness in $H^s$}\label{sec5}
This final section is devoted to the proof of Theorem \ref{Thm2}.
The proof based on Proposition \ref{Prop1} and Lemma \ref{le2}
is similar to the one given in the previous section for the $L^2$ case; therefore, we shall give only a sketch of it.

\subsection{The subcritical case}
For appropriate values of $T,M>0$, we show that $\Phi$ defines a contraction map on
\begin{align*}
X(T,M)=\Big\lbrace u\in  &C_t(I;H_x^s)\cap L_t^q(I;L_x^r(|x|^{-r\gamma})) :\\
&\sup_{t\in I}\|u\|_{H_x^s}+\|u\|_{L_t^q(I;L_x^r(|x|^{-r\gamma}))} \leq  M \Big\rbrace
\end{align*}
equipped with the distance
$$d(u,v)=\sup_{t\in I}\|u-v\|_{H_x^s}+\|u-v\|_{L_t^q(I;L_x^r(|x|^{-r\gamma}))}.$$
Here, $I=[0,T]$ and the exponent pair $(q,r,\gamma)$ is given as in Theorem \ref{Thm2}.
Firstly, $\Phi(u) \in X$ for $u \in X$, i.e.,
\begin{equation}\label{fgjs}
\sup_{t\in I}\|\Phi(u)\|_{H_x^s}+\|\Phi(u)\|_{L_t^q(I;L_x^r(|x|^{-r\gamma}))} \leq M.
\end{equation}
Indeed, using \eqref{TT*} with $(\tilde{q}, \tilde{r}, \tilde{\gamma}) = (\tilde{q}_1, \tilde{r}_1, \tilde{\gamma}_1)$
followed by \eqref{F1} in Lemma \ref{le2} and \eqref{T}, we see that
\begin{align}
\nonumber  \| \Phi(u)\|_{L_t^{q}(I ; L_{x}^{r}(|x|^{-r \gamma}))}
  &\leq C\|u_0\|_{H^s} + C\left\||x|^{-\alpha}|u|^{\beta}u \right\|_{L_{t}^{\tilde{q}_1'} (I; L_{x}^{\tilde{r}_1'}(|x|^{\tilde{r}_1' \tilde{\gamma}_1}))}\\
\nonumber&\leq C\| u_0 \|_{{H}^s} +  CT^{\theta_1} \left\| u \right\|_{L_t^{q}(I; L_{x}^{r} (|x|^{-r\gamma}))} ^{\beta+1} \\
\label{1js}&\leq C\| u_0 \|_{{H}^s} +  CT^{\theta_1} M^{\beta+1}
\end{align}
because $\| f \|_{\dot{H}^s}\lesssim\| f \|_{H^s}$.
On the other hand, from Plancherel's Theorem combined with the smoothing estimate \eqref{T*}, we see
\begin{align*}
\sup_{t\in I}\|u\|_{H_x^s} &\leq C\|u_0\|_{H^s}  + C\sup_{t\in I} \left\| \int_{0}^t e^{-i\tau\Delta} F(u)\ d\tau \right\|_{H^s}\\
&\leq C\|u_0\|_{H^s}  +C\big(\| F(u)\|_{L_t^{\tilde{q}_1'}(I; L_x^{\tilde{r}_1'}(|x|^{\tilde{r}_1'\tilde{\gamma}_1}))}+\big\| |\nabla|^{-s} F(u) \big\|_{L_{t}^{{\tilde{q}_2}'}(I;L_x^{{\tilde{r}_2}'}(|x|^{{\tilde{r}_2}' \tilde{\gamma}_2}))}\big).
\end{align*}
Here we also used, for the second inequality, that $\| f \|_{H^s}\lesssim\| f \|_{\dot{H}^s}+\| f \|_{L^2}$.
By using the nonlinear estimates \eqref{F1} and \eqref{F2}, it follows then that
\begin{align}\label{cr}
\nonumber\sup_{t\in I}\|u\|_{H_x^s} &\leq  C\| u_0\|_{H^s} +C (T^{\theta_1}+T^{\theta_2}) \|u\|_{L_t^q(I; L_x^r(|x|^{-r\gamma}))}^{\beta+1}\\
 &\leq  C\| u_0\|_{H^s} +C (T^{\theta_1}+T^{\theta_2}) M^{\beta+1}.
\end{align}
Thus, if we set $M=4C\| u_0 \|_{H^s}$ and take $T>0$ such that
\begin{equation}\label{dgh1}
 C(T^{\theta_1}+T^{\theta_2})M^{\beta}\leq\frac 18,
\end{equation}
we obtain \eqref{fgjs}.
Nextly, $\Phi$ is a contraction on $X$, i.e., for $u,v\in X$
 \begin{equation}\label{alks}
 d(\Phi(u),\Phi(v)) \leq \frac{1}{2} d(u,v).
 \end{equation}
Indeed, we first see as above that
\begin{align*}
\sup_{t\in I}\|\Phi(u)-\Phi(v)\|_{H_x^s}
\leq\ &  C \left\| F(u)-F(v) \right\|_{L_t^{\tilde{q}_1'}(I; L_x^{\tilde{r}_1'}(|x|^{\tilde{r}_1'\tilde{\gamma}_1}))} \cr
& +  C\left\| |\nabla|^{-s} (F(u)-F(v)) \right\|_{L_{t}^{{\tilde{q}_2}'}(I;L_x^{{\tilde{r}_2}'}(|x|^{{\tilde{r}_2}' \tilde{\gamma}_2}))}.
\end{align*}
Using the simple inequality \eqref{sin}, the right-hand side here is bounded as in Lemma \ref{le2} by
$$2C(T^{\theta_1}+T^{\theta_2}) M^{\beta}\left\| u-v \right\|_{L_t^{q}(I; L_{x}^{r}(|x|^{-r \gamma}))}.$$
Meanwhile, by \eqref{TT*} we have
\begin{equation*}
\|\Phi(u)-\Phi(v)\|_{L_{t}^{q}(I; L_{x}^{r}(|x|^{-r \gamma}))}
\leq C\|F(u)-F(v)\|_{L_t^{\tilde{q}_1'}(I; L_{x}^{\tilde{r}_1'} (|x|^{\tilde{r}_1' \tilde{\gamma}_1}))}.
	\end{equation*}
Consequently, we obtain \eqref{alks} because of \eqref{dgh1}.

We have proved that there exists a unique solution
$u \in C(I ; H^s) \cap L^{q}(I ; L^{r}(|x|^{-r \gamma}))$ with $T=T(\|u_0\|_{H^s}, \alpha, \beta)$.
The continuous dependence on the data follows clearly as before.

\subsection{The critical case}
The situation in this case is similar to the $L^2$ case because
the power $\theta_1$ in the above argument becomes zero in this case.
But the other power $\theta_2=s/2>0$ yields a slight different conclusion as shown in the proof.

We start from showing that $\Phi$ defines a contraction on
\begin{align*}\label{contsets}
\widetilde{X}(T,M,N)=\Big\lbrace u\in  &C_t(I;H^s)\cap L_t^q(I;L_x^r(|x|^{-r\gamma})) :\\
&\sup_{t\in I}\|u\|_{H_x^s}\leq N,\quad \|u\|_{L_t^q(I;L_x^r(|x|^{-r\gamma}))} \leq M \Big\rbrace
\end{align*}
equipped with the distance
$$d(u,v)=\sup_{t\in I}\|u-v\|_{H_x^s}+\|u-v\|_{L_t^q(I;L_x^r(|x|^{-r\gamma}))}.$$
As in \eqref{cr} and \eqref{1js}, we see that
\begin{equation*}
\sup_{t\in I}\|\Phi(u)\|_{H_x^s}\leq C\|u_0\|_{H^s}+C(1+T^{\theta_2}) M^{\beta+1}
\end{equation*}
and
$$\|\Phi(u)\|_{L_t^q(I;L_x^r(|x|^{-r\gamma}))}\leq \|e^{it\Delta}u_0\|_{L_t^q(I;L_x^r(|x|^{-r\gamma}))}+C M^{\beta+1},$$
respectively.
By the dominated convergence theorem we observe that
\begin{equation}\label{oax}
\|e^{it\Delta}u_0\|_{L_t^q(I;L_x^r(|x|^{-r\gamma}))}\leq \varepsilon
\end{equation}
for given $\varepsilon>0$ which will be chosen later, if $T>0$ is small enough.
Hence, $\Phi(u)\in \widetilde{X}$ for $u\in\widetilde{X}$ if
\begin{equation}\label{azs}
C\|u_0\|_{H^s}+C (1+T^{\theta_2})M^{\beta+1}\leq N\quad\text{and}\quad \varepsilon+C M^{\beta+1}\leq M.
\end{equation}
Meanwhile, using the same argument employed to show \eqref{alks}, we see
\begin{align*}
d(\Phi(u),\Phi(v))
&=\sup_{t\in I}\|\Phi(u)-\Phi(v)\|_{H_x^s}+\|\Phi(u)-\Phi(v)\|_{L_t^q(I;L_x^r(|x|^{-r\gamma}))}\\
&\leq 2C(1+T^{\theta_2}) M^\beta d(u,v).
\end{align*}
Now we take $N=2C\|u_0\|_{H^s}$ and $M=2\varepsilon$, and then choose $\varepsilon>0$ small enough so that
\eqref{azs} holds and $2C(1+T^{\theta_2})M^\beta \leq1/2$.
It follows now that $\widetilde{X}$ is stable by $\Phi$ and $\Phi$ is a contraction on $\widetilde{X}$.

Therefore, there exists a unique solution $u \in C(I ; H^s) \cap L^{q}(I ; L^{r}(|x|^{-r \gamma}))$
in the time interval $[0,T]$ with $T=T(u_0,\alpha,\beta)$.
Here we cannot take $T=\infty$ even if $\|u_0\|_{H^s}$ is small enough to guarantee \eqref{oax},
because $\theta_2=s/2>0$.
The continuous dependence on the data follows clearly in the same way as before.


\begin{thebibliography}{9}


\bibitem{BK} M. Ben-Artzi and S. Klainerman, \textit{Decay and regularity for the Schr\"odinger equation},
J. Anal. Math. 58 (1992), 25-37.

	\bibitem{B} L. Berg\'e, \textit{Soliton stability versus collapse}, Phys. Rev. E, 62 (2000), 3071-3074.

\bibitem{BL} J. Bergh and J. L\"{o}fstr\"{o}m, \textit{Interpolation Spaces, An Introduction}, Springer-Berlag, Berlin-New York, 1976.

	\bibitem{C} T. Cazenave, \textit{Semilinear Schr\"odinger equations}, Courant Lecture Notes in Mathematics, 10. New York University, Courant Institute of Mathematical Sciences, New York; American Mathematical Society, Providence, RI, 2003.

	\bibitem{CK} M. Christ and A. Kiselev, \textit{Maximal functions associated to filtrations}, J. Funct. Anal. 179 (2001), 409-425.

	\bibitem{CW} T. Cazenave and F. B. Weissler, \textit{Some remarks on the nonlinear Schr\"odinger equation in the critical case}, Nonlinear Semigroups, Partial differential equations and attractors (Washington, DC, 1987), Lecture Notes in Math. 1394, Springer, Berlin, 1989, 18-29.

	\bibitem{CW2} T. Cazenave and F. B. Weissler, \textit{The Cauchy problem for the critical nonlinear Schr\"{o}dinger equation in $H^{s}$}, Nonlinear Anal., 14 (1990), 807-836.

	\bibitem{D} V. D. Dinh, \textit{Scattering theory in a weighted $L^2$ space for a class of the defocusing inhomogeneous nonlinear Schr\"odinger equation}, to appear in Adv. Pure Appl. Math.

	\bibitem{Fa} L. G. Farah, \textit{Global well-posedness and blow-up on the energy space for the inhomogeneous nonlinear Schr\"odinger equation}, J. Evol. Equ. 16 (2016), 193-208.

	\bibitem{Ge} F. Genoud, \textit{An inhomogeneous, $L^2$-critical, nonlinear Schr\"odinger equation}, Z. Anal. Anwend. 31 (2012), 283-290.

	\bibitem{GS} F. Genoud and C. A. Stuart, \textit{Schr\"odinger equations with a spatially decaying nonlinearity: existence and stability of standing waves}, Discrete Contin. Dyn. Syst. 21 (2008), 137-186.

	\bibitem{GU} C. M. Guzm\'{a}n, \textit{On well posedness for the inhomongeneous nonlinear Schr\"{o}dinger equation}, Nonlinear Anal. Real World Appl. 37 (2017), 249-286.

	\bibitem{GV3} J. Ginibre and G. Velo, \textit{On a class of nonlinear Schr\"{o}dinger equations I. The Cauchy problem, general case}, J. Funct. Anal. 32 (1979), 1-32.

\bibitem{GV4}  J. Ginibre and G. Velo, \textit{The global Cauchy problem for the nonlinear Schr\"{o}dinger equation revisited}, Ann. Inst. H. Poincar\'{e} Anal. Non Lin\'{e}aire, 2 (1985), 309-327.

\bibitem{HR} J. Holmer, S. Roudenko, \textit{A sharp condition for scattering of the radial 3D cubic nonlinear Schr\"{o}dinger equation}, Comm. Math. Phys. 282 (2008), 435-467.

\bibitem{K} T. Kato, \textit{On nonlinear Schr\"{o}dinger equations}, Ann. Inst. H. Poincar\'{e} Phys. Th\'{e}or. 46 (1987), 113-129.


\bibitem{KY} T. Kato and K. Yajima, \textit{Some examples of smooth operators and the associated smoothing effect},
Rev. Math. Phys. 1 (1989), 481-496.

\bibitem{KT} M. Keel and T. Tao, \textit{Endpoint Strichartz estimates}, Amer. J. Math 120 (1998), 955-980.
	
\bibitem{SW} E. M. Stein and G. Weiss, \textit{Fractional Integrals on $n$-Dimensional Euclidean Space}, J. Math. Mech. 7 (1958), 503-514.
	
\bibitem{St} R. S. Strichartz, \textit{Restrictions of Fourier transforms to quadratic surfaces and decay of solutions of wave equations},
Duke Math. J. 44 (1977), 705-714.

\bibitem{Su} M. Sugimoto, \textit{Global smoothing properties of generalized Schr\"odinger equations},
J. Anal. Math. 76 (1998), 191-204.

\bibitem{TM} I. Towers and B. A. Malomed, \textit{Stable (2+1)-dimensional solitons in a layered medium with sign-alternating Kerr nonlinearity}, J. Opt. Soc. Amer. B Opt. Phys. 19 (2002), 537-543.

\bibitem{Ts} Y. Tsutsumi, \textit{ $L^{2}$-solutions for nonlinear Schr\"{o}dinger equations and nonlinear groups}, Funkcial. Ekvac. 30 (1987), 115-125.

\bibitem{V} M. C. Vilela, \textit{Regularity of solutions to the free Schr\"odinger equation with radial initial data},
Illinois J. Math. 45 (2001), 361-370.

\bibitem{W} K. Watanabe, \textit{Smooth perturbations of the selfadjoint operator $|\Delta|^{\alpha/2}$}, Tokyo J. Math. 14 (1991), 239-250.
	
	\bibitem{We} M. Weinstein, \textit{Nonlinear Schr\"odinger equations and sharp interpolation estimates}, Comm. Math. Phys. 87 (1983), 567-576.
	
\end{thebibliography}
\end{document}